\documentclass[a4paper,11pt,leqno]{amsart}
\usepackage{latexsym}

\usepackage[width=5.65in,height=7.85in,centering]{geometry}
\usepackage[all]{xy}

\usepackage{amssymb} 
\usepackage{amsmath} 
\usepackage{bbm}
\usepackage{enumitem}  

\usepackage{tikz}
\usetikzlibrary{arrows,cd}  
\usepackage{amsthm}

\def\P{{\mathbb{P}}}
\def\Z{{\mathbb{Z}}}
\def\K{{\mathbb{K}}}
\def\CC{{\mathbb{C}}}
\def\A{{\mathcal{A}}}

\def\OO{{\mathcal{O}}} 
\def\oS{{\overline{S}}}  

\def\pp{{\mathfrak p}}  
\def\one{{\mathbbm1}}   

\newcommand{\abs}[1]{\left|#1\right|}     
\newcommand{\set}[1]{\left\{#1\right\}}   

\newcommand{\noproof}{\popQED{\ensuremath{\Box}}}

\DeclareMathOperator{\rank}{rank}
\DeclareMathOperator{\codim}{codim}

\DeclareMathOperator{\Der}{Der}

\DeclareMathOperator{\pd}{pd}

\DeclareMathOperator{\spec}{spec}
\DeclareMathOperator{\res}{res}

\DeclareMathOperator{\POexp}{POexp}

\DeclareMathOperator{\Hom}{Hom}
\DeclareMathOperator{\Ext}{Ext}
\DeclareMathOperator{\NF}{NF}   
\DeclareMathOperator{\ND}{ND}   
\DeclareMathOperator{\NO}{N\Omega}   
\DeclareMathOperator{\supp}{supp}  
\DeclareMathOperator{\Proj}{Proj}

\numberwithin{equation}{section}

\newtheorem{theorem}{Theorem}[section]
\newtheorem{prop}[theorem]{Proposition}
\newtheorem{cor}[theorem]{Corollary}
\newtheorem{lemma}[theorem]{Lemma}

\theoremstyle{remark}
\newtheorem{rem}[theorem]{Remark}
\newtheorem{example}[theorem]{Example}
\theoremstyle{definition}
\newtheorem{define}[theorem]{Definition}

\title{Deletion-restriction for logarithmic forms on multiarrangements}  
\author{Takuro Abe}
\address{  
  Department of Mathematics, Rikkyo University}
\email{abetaku@rikkyo.ac.jp}
\author{Graham Denham}
\address{Department of Mathematics, University of Western Ontario}
\email{gdenham@uwo.ca}
\date{\today}

\begin{document}

\begin{abstract}
  We consider the behaviour of logarithmic differential forms on
  arrangements and multiarrangements
  of hyperplanes under the operations of deletion and
  restriction, extending early work of G\"unter Ziegler~\cite{Z2}.  The
  restriction map of logarithmic forms to a hyperplane
  is not necessarily surjective, 
  and we measure the failure of surjectivity in terms of commutative algebra
  of logarithmic forms and derivations.  We find that the dual notion of
  restriction of logarithmic vector fields behaves similarly but
  inequivalently.  A main result is that, if an arrangement is free, then
  any arrangement obtained by adding a hyperplane has the ``dual strongly
  plus-one generated'' property. One application is another proof of
  a main result of \cite{A6} characterizing
  when adding a hyperplane to a free arrangement remains free.  A further
  application is to resolve two conjectures due to Ziegler, which we
  defer to a sequel~\cite{AD23b}.
\end{abstract}


\maketitle

\section{Introduction}
Let $V=\K^\ell$ and 
$S=\K[x_1,\ldots,x_\ell]$ its coordinate ring. Let
$\Der S:=\oplus_{i=1}^\ell S \partial_{x_i}$, 
$\Omega^1_V:=
\oplus_{i=1}^\ell S dx_i$, and $\Omega^p_V=\wedge^p \Omega^1_V$. 
Let $\A$ be a central hyperplane arrangement: that is,
a finite set of linear hyperplanes in $V$.  For each $H \in \A$ fix 
$\alpha_H \in V^*$ such that 
$\ker \alpha_H=H$, and 
let $Q(\A):=
\prod_{H \in \A} \alpha_H$.
Then \emph{(the module of) logarithmic vector fields} of $\A$ is defined to be
\[
D(\A):=\{\theta \in \Der S \mid \theta(\alpha_H) \in S\alpha_H\ (\forall H \in \A)\},
\]
and \emph{(the module of) logarithmic differential ($1$-)forms} is
\[
\Omega^1(\A):=
\{
\omega \in \displaystyle \frac{1}{Q(\A)} \Omega^1_V 
\mid 
\omega \wedge d\alpha_H \in 
\displaystyle \frac{\alpha_H}{Q(\A)} \Omega^2_V \ 
(\forall H \in \A)\}.
\]
These graded $S$-modules are mutually dual and reflexive of rank
$\ell$ (see \cite{OT}), but in general they are not free.
We say that $\A$ is \emph{free} with
$\exp(\A)=(d_1,d_2,\ldots,d_\ell)$ provided that
\[
D(\A) \cong \oplus_{i=1}^\ell S[-d_i]
\quad\text{or, equivalently,}
\quad
\Omega^1(\A) \cong \oplus_{i=1}^\ell S [d_i].
\]

The study of free arrangements has an extensive history: see, for example,
\cite{A2,A4,A5,A6} for recent results and references.  One of the subtleties
of the theory is that deletion-contraction arguments, which are based on
both removing and restricting to a hyperplane, generally work well
only with additional hypotheses on freeness and degrees of generation.
For example,
Terao's Addition-Deletion Theorem~\cite{T1} states that, if two of
$\A$, $\A'$ and $\A^H$ are free, so is the third, provided a condition on
the exponents is satisfied.  Here, for a choice of $H\in\A$, we let
$\A':=\A-\set{H}$, and
\[
\A^H := \set{L\cap H\mid L\in\A'},
\]
an arrangement in $H\subseteq V$.  

At the heart of Terao's argument there is the \emph{Euler exact sequence}
\begin{equation}\label{eq:euler}
\begin{tikzcd}[column sep=small]
  0\ar[r] & D(\A') \ar[r,"\cdot\alpha_H"] & D(\A)\ar[r,"\rho^H"] &
  D(\A^H),
\end{tikzcd}
\end{equation}
where $\rho^H$ is induced by the restriction map $S\to S/(\alpha_H)$.  Under
his hypotheses, the map $\rho^H$ is surjective, and logarithmic derivations
on $\A$ can be understood in terms of those on the smaller arrangements
$\A'$ and $\A^H$.  In general, though, the sequence \eqref{eq:euler} fails
to be right exact.  We show in Theorem~\ref{thm:les} how to extend it
to a long exact sequence.

The analogous exact sequence of logarithmic forms was first considered by
Ziegler in~\cite{Z2}:
\begin{equation}\label{eq:euler'}
\begin{tikzcd}[column sep=small]
  0\ar[r] & \Omega^1(\A) \ar[r,"\cdot\alpha_H"] & \Omega^1(\A')\ar[r,"i^*_H"] &
  \Omega^1(\A^H).
\end{tikzcd}
\end{equation}
The first map is given by multiplication by $\alpha_H$, and the second by
restricting a one-form along the inclusion of the hyperplane,
$i_H\colon H\to V$.  Once again, the sequence \eqref{eq:euler'} is right exact
only under hypotheses that we elucidate in this paper.
Since $D(\A)$ and $\Omega^1(\A)$ are dual,
the two sequences have a similar flavour but also obvious differences.
For example, the first author shows in \cite{A5} that, if $\A$ is free,
then for any $H \in \A$, that $\pd_S D(\A')\leq 1$: see 
Theorem \ref{A5}. On the other hand, even if $\A$ 
is free, there are examples for which $\pd_S \Omega^1(\A')=\ell-2$,
which equals the upper bound imposed by reflexivity: for example, if $\A$
is a rank $\ell=3$ arrangement, any non-free deletion $\A'$ has this property.

\subsection*{Multiarrangements}
The arguments we make in this article work in the somewhat more general
context of multiarrangements, where each hyperplane $H$ is assigned a
nonnegative integer multiplicity, $m(H)$.  This is natural from various
points of view: we refer to \cite{Z}, \cite{ATW2}, and to \S\ref{sec:restriction} for
definitions.  For any arrangement $\A$ with multiplicity function
$m\in\Z^{\A}$, Terao, Wakefield and the first author~\cite[Prop.\ 2.2]{ATW2} showed
that the sequence \eqref{eq:euler} generalizes to a sequence
\begin{equation}\label{eq:m_euler}
\begin{tikzcd}[column sep=small]
  0\ar[r] & D(\A,m') \ar[r,"\cdot\alpha_H"] & D(\A,m)\ar[r,"\rho^H"] &
  D(\A^H,m^*),
\end{tikzcd}
\end{equation}
where $m'$ is obtained from $m$ by lowering the multiplicity of $H$ by one,
and $m^*$ is the \emph{Euler multiplicity} they introduce
(Definition~\ref{def:euler mult}).

In Section~\S\ref{sec:restriction}, we show that the corresponding 
sequence for logarithmic forms exists and
generalizes \eqref{eq:euler'} in the same way:

\begin{prop}\label{main1}  
For any $H\in\A$, there is an exact sequence
\begin{equation}
\begin{tikzcd}[column sep=small]
  0\ar[r] & \Omega^1(\A,m) \ar[r,"\cdot\alpha_H"] & \Omega^1(\A,m')
  \ar[r,"i^*_H"] & \Omega^1(\A^H,m^*).
\end{tikzcd}
  \label{E1}
\end{equation}
\label{eq:m_euler'}
\end{prop}

For multiarrangements, we obtain the following conditions for right exactness.
We generalize this result to multiarrangements as follows.  Once again,
let $(\A,m)$ be a multiarrangement and let $m'$ be the vector
obtained from $m$ by lowering the multiplicity of $H$ by one.

\begin{theorem}[Free surjection theorem]
  Suppose that $(\A,m')$ is free for some hyperplane $H\in\A$.  Then 
  \[
  \rho^H\colon D(\A,m) \rightarrow D(\A^H,m^*)
  \]
is surjective. 
\label{thm:FST}
\end{theorem}

In the multiplicity-free case, this was established earlier by
the first author~\cite{A9}.  The dual version is analogous:
\begin{theorem}[Free surjection theorem 2]
  Suppose that $(\A,m)$ is free. Then
  \[
  i_H^*\colon \Omega^1(\A, m') \rightarrow \Omega^1(\A^H,m^*)
  \]
  is surjective for all $H \in \A$.
\label{thm:FST2}
\end{theorem}
In \S\ref{sec:LES} we refine these results using duality.  For example,
if $C_D$ denotes the cokernel of the injective map
\[
D(\A,m')\stackrel{\alpha_H}{\rightarrow} D(\A),
\]
we show in Proposition~\ref{prop:coker_dual} that the dual of $C_D$ always
equals $\Omega^1(\A^H,m^*)$.  This leads to a long exact sequence
extending \eqref{eq:m_euler}.  Parallel results exchange the roles
of logarithmic derivations and forms.

\subsection*{SPOG properties}

In \S\ref{sec:SPOG}, we consider arrangements which are, from the point
of view of logarithmic derivations, the next simplest after free arrangements.
These are arrangements for which $D(\A)$ is generated by $\ell+1$ elements
which are subject to a single relation: see Definition~\ref{SPOGdef} for the
precise definition of \emph{strongly plus-one generated} (SPOG) arrangements.
The first author~\cite{A5} showed the following:
\begin{theorem}[Theorem 1.13, \cite{A5}]
  Suppose $\A$ is a free arrangement.  Then each deletion $\A'$ is either
  free or SPOG.
  \label{A5}
\end{theorem}

Here, we generalize this result to multiarrangements (Theorem~\ref{multimain}).
The analogous statement for logarithmic forms is known to be false, by
Edelman--Reiner's famous counterexample to Orlik's conjecture: we refer to
\cite{A5} for details.  In this paper, we introduce a condition which we call
\emph{dual SPOG} (Definition~\ref{dualSPOG}), which turns out to be the
appropriate replacement for ``strongly plus-one generated'':

\begin{theorem}
  Let $\A'$ be free. Then $\A$ is free or dual SPOG.
\label{SPOGmain}  
\end{theorem}

The multiarrangement version of this result appears as the second part of
Theorem~\ref{multimain}.  As an application, we give a short new proof of
the combinatorial addition-deletion theorem of the first author~\cite{A6}.

\medskip

\noindent
\textbf{Acknowledgements}. 
The first author is
partially supported by JSPS KAKENHI Grant Numbers JP18KK0389 and JP21H00975.
The second author is supported by NSERC of Canada and would like to thank
the Max Planck Institute for Mathematics for its hospitality during the final
preparation of this manuscript.  We are grateful to an anonymous referee
for a careful reading that led to numerous improvements.

\section{Restrictions of forms and surjectivity}
\label{sec:restriction}

In this section, we prove Proposition \ref{main1}. 
For that, first let us recall the
notation used in Proposition \ref{main1}. 
A \emph{multiarrangement} is the pair $(\A,m)$, where $\A$ is an 
arrangement and $m$ is a map $m:\A \rightarrow \Z_{>0}$.  One may think of a
multiarrangement as a linear realization of a matroid without loops in which
parallel elements are regarded as unordered: then $\A$ is a realization of
the underlying simple matroid, and the multiplicities $m(H)$ keep track of
the size of each equivalence class of parallel elements.

Let $Q(\A,m):=
\prod_{H \in \A} \alpha_H^{m(H)}$ and 
$|m|:=\deg Q(\A,m)$. Then \emph{the module of logarithmic derivations}
is defined to be
\[
D(\A,m):=\{\theta \in \Der S \mid \theta(\alpha_H) \in S\alpha_H^{m(H)}\ (\forall H \in \A)\},
\]
and \emph{(the module of) logarithmic differential ($1$-)forms} is
\[
\Omega^1(\A,m):=
\{
\omega \in \displaystyle \frac{1}{Q(\A,m)} \Omega^1_V 
\mid 
\omega \wedge d\alpha_H \in 
\displaystyle \frac{\alpha_H^{m(H)}}{Q(\A,m)} \Omega^2_V \ 
(\forall H \in \A)\}.
\]
They are again $S$-dual modules, and freeness and exponents are defined in the
same way as for
$\A$. Note that $\A$ can be regarded as a multiarrangement $(\A,\one)$.

Let us define the Euler exact sequence for multiarrangements.
For this, we recall that Euler restriction has a nontrivial generalization
in the presence of multiplicities.

\begin{define}[\cite{ATW2}]\label{def:euler mult}
Let $(\A,m)$ be a multiarrangement, $H \in \A$ and let 
$\delta_H:\A \rightarrow \{0,1\}$ be the characteristic multiplicity of $H$, i.e., 
$\delta_H(L)=1$ if $H=L$, and $0$ otherwise. Define 
$m':=m-\delta_H$. Let $X \in \A^H$. Then $\A_X =\A_X^e \times \emptyset_{\ell-2}$, where $\A_X^e$ is the essentialization of 
$\A_X$ and $\emptyset_{\ell-2}$ is the empty arrangement 
in $\K^{\ell-2}$.
Thus $D(\A_X^e,m_X)$ has a basis 
$\theta_X,\varphi_X$ such that $\alpha_H \nmid \theta_X$ 
and $\ \alpha_H \mid \varphi_X$. 
In this terminology, define 
the \emph{Euler multiplicity} $m^*:\A^H \rightarrow \Z_{>0}$ by 
$$
m^*(X):=\deg \theta_X.
$$
We say that $(\A^H,m^*)$ is the \emph{Euler 
restriction} of $(\A,m)$ onto $H$.
\label{multiEuler}
\end{define}

\begin{rem}
When $m \equiv 1$, then it is easy to see that 
$m^*\equiv 1$.
\end{rem}

Now let us prove the exactness of the sequence \eqref{E1}:

\begin{proof}[Proof of Proposition~\ref{main1}.]
Let $H:x_1=0$ and $\{x_1=x_2=0\} =X \in \A^H$. It suffices to show that, for any $\omega \in \Omega^1(\A,m')$, the pole order of $i_H(\omega) $ along $x_2=0$ is at most $m^*(X)=\deg \theta$. 
Let $Q_1:=Q(\A_X,m_X')$ and $Q_2:=Q(\A,m')/Q_1$. Then 
$Q_2 \omega \in \Omega^1(\A_X,m'_X)$. 
Since the pole order of hyperplanes containing $X$, i.e., that contributes the pole order of $i_H(\omega)$ along $X$, is stable under the multiplication of $Q_2$, it suffices to show that $i_H(\eta)$ has poles along $X$ at most $m^*(X)$. Let $\Omega^1(\A_X^e,m_X')=\langle \eta_1,\eta_2\rangle_S$ be a homogeneous basis. Since $\Omega^1(\A_X^e,m_X)$ is free too, we may assume that $\Omega^1(\A_X^e,m_X)=\langle 
\eta_1/\alpha_H,\eta_2\rangle$. So by definition, $i_H(\eta_1)=0$ and $i_H(\eta_2) \in \K \frac{1}{x_2^s}dx_2$, where $s=-\deg \eta_2$. Hence the maximal pole order of $i_H(\eta)$ along $X$ coincides with $s=|\deg \eta_2|$. By the duality of $D(\A)$ and $\Omega^1(\A)$, it holds that $\deg \theta_X=-\deg \eta_2=s$ and $\deg \psi=-(\deg \eta_1-1)$. Hence $|\deg \eta_2|=s=m^*(X).$
\end{proof}

Recall that
\begin{eqnarray*}
D^p(\A,m):&=&\{ \theta \in \wedge^p \Der S \mid 
\theta(\alpha_H,f_2,\ldots,f_p) \in S \alpha_H^{m(H)}\ (\forall H \in \A,\ 
\forall f_2,\ldots,f_p \in S)\},\\
\Omega^p(\A,m):&=&\{ \omega \in \frac{1}{Q(\A,m)}\Omega^p_V \mid 
(Q(\A,m)/\alpha_H^{m(H)}) \omega \wedge d\alpha_H \in \Omega^1_V\ (\forall H \in \A)\}.
\end{eqnarray*}

The following isomorphism is well-known when $m=\one$: see, e.g., \cite{Y3}.
\begin{lemma}
  For all $0\leq p\leq \ell$, we have
  \[
  \Omega^p(\A,m) \cong D^{\ell-p}(\A,m).
  \]
\label{ID}
\end{lemma}
\begin{proof}
For $I \subset [\ell]:=\{1,\ldots,\ell\}$, let 
$I^c:=\{i \in [\ell] \mid i \not \in I\}$. Define 
\[
dx_I:=\wedge_{i \in 
I} dx_i,\ \partial_I:=\wedge_{ i \in I} \partial_{x_i},
\]
and $[\ell]_p:=\{I \in [\ell] \mid \abs{I}=p\}$.
Then the map
\[
\Omega^p(\A,m) \ni \sum_{I \in [\ell]_p}
(f_I/Q(\A,m)) dx_I
\mapsto \sum_{I \in [\ell]_p} f_I \partial_{I^c} 
\in
D^{\ell-p}(\A,m)
\]
is easily seen to be an isomorphism.
\end{proof}

We introduce Euler sequences in higher degrees as follows:

\begin{prop}\label{HE1}\label{HE2}
For any $p \ge 0$, there are exact sequences
\begin{align}
&0 \rightarrow D^p(\A,m-\delta_H) \stackrel{\cdot \alpha_H}{\rightarrow } D^p(\A,m) 
\stackrel{\rho^H}{\rightarrow} D^p(\A^H,m^*) \quad \text{and} \label{eq:HE1D}\\
&0 \rightarrow \Omega^p(\A,m) \stackrel{\cdot \alpha_H}{\rightarrow } \Omega^p(\A,m-\delta_H) 
\stackrel{i_H^*}{\rightarrow}\Omega^p(\A^H,m^*). \label{eq:HE1O}
\end{align}
\end{prop}

\begin{proof}
  The first statment is proven in \cite[Prop.\ 4.10]{M}, and the argument
  for the second is similar to that of Proposition~\ref{main1}.
  We must show that $i_H^*(\omega) \in \Omega^p(A^H,m^*)$ for any $\omega \in 
\Omega^p(\A,m-\delta_H)$. Let $\alpha_H=x_1$. 
Let $X =\{x_1=x_2=0\}\in \A^H$, $Q_1:=Q(\A_X,(m-\delta_H)_X)$ and 
$Q_2:=Q(\A,m-\delta_H)/Q_1$. Then 
$Q_2 \omega \in \Omega^1(\A_X,(m-\delta_H)_X)=\langle
\omega_1^X,\alpha_H \omega_2^X,d{x_3},\ldots,d{x_\ell}\rangle_S$.
Thus $i_H^*(\omega)$ is in the image of $i_H^* (\omega_1^X)$, which is of the form 
$(\alpha_X)^{-\deg \theta_X} \partial_{x_2}$. Hence $i_H^* (\omega)$ has a pole of order 
$-\deg \omega_1^X=m^*(X)$, which completes the proof.
\end{proof}

We will use the fact that logarithmic derivations and forms behave well under
localization, following \cite[\S4.6]{OT}.  If $\pp$ is a prime ideal of $S$,
let $X(\pp)$ be the intersection of all hyperplanes of $\A$ containing $V(\pp)$.
Then
\begin{equation}\label{eq:localize}
  \Omega^p(\A,m)_{\pp}=\Omega^p(\A_X,m_X)_{\pp}\text{~for all $p\geq0$},
\end{equation}
where $X=X(\pp)$,
and analogously for derivations.  In particular, the localization
$\Omega^1(\A,m)_{\pp}$ is a free $S_{\pp}$-module if and only if the
multiarrangement $(\A_X,m_X)$ is free.

To prove Theorem \ref{thm:FST2}, we need the following.

\begin{prop}
\[
H_*^0(\widetilde{\Omega^p(\A,m)}):=
\bigoplus_{j \in 
\Z} H^0(\widetilde{\Omega^p(\A,m)}(j))=
\Omega^p(\A,m),
\]
and
\[
H_*^0(\widetilde{D^p(\A,m)}):=
\bigoplus_{j \in 
\Z} H^0(\widetilde{D^p(\A,m)}(j))=
D^p(\A,m).
\]
\label{gs}
\end{prop}
\begin{proof}
  The modules $\Omega^p(\A,m)$ and $D^p(\A,m)$ are reflexive, so their
  depth is at least two~\cite[Prop.\ 1.3]{Har80}, and
  the conclusion follows by comparing with local cohomology: see, e.g.,
  \cite[Cor.\ A1.13]{Eisensyz}.
\end{proof}

Now we can prove Theorem \ref{thm:FST2}.
\medskip

\begin{proof}[Proof of Theorem \ref{thm:FST2}.]
The statement is clear if $\ell \le 2$.  We argue by induction on rank, using
\eqref{eq:localize}.
For $\ell \ge 3$, this gives an exact sequence of sheaves on $\P^{\ell-1}$, 
\[
0 \rightarrow \widetilde{\Omega^1(\A,m)} \stackrel{\cdot \alpha_H}{\rightarrow}\widetilde{\Omega^1(\A,m-\delta_H) }
\stackrel{i_H^*}{\rightarrow} \widetilde{\Omega^1(\A^H,m^*)}
\rightarrow 0.
\]
Now we write the long exact sequence in cohomology, with the help of
Proposition \ref{gs}.
By hypothesis, $(\A,m)$ is free, so by Proposition~\ref{gs},
the vector bundle $\widetilde{\Omega^1(\A,m)}$ is split.  It follows that 
\[
H^1_*(\P^{\ell-1},\widetilde{\Omega^1(\A,m)})=0,
\]
since $\ell \geq 2$, which implies
\[
\begin{tikzcd}[column sep=small]
  0\ar[r] & \Omega^1(\A,m) \ar[r,"\cdot\alpha_H"] & \Omega^1(\A,m-\delta_H)
  \ar[r,"i^*_H"] & \Omega^1(\A^H,m^*)\ar[r] & 0
\end{tikzcd}
\]
is exact.
\end{proof}
 
The same argument, using the fact that 
$\Omega^p(\A,m) =\wedge^p \Omega^1(\A,m)$ and 
$D^p(\A,m) \cong \wedge^p D(\A,m)$ when $\A$ is free (see \cite{OT} for 
example), extends Theorem \ref{thm:FST2}:

\begin{cor}
  For all $p\geq1$,
\begin{itemize}
    \item [(1)]
    The Euler restriction map $\rho^p:D^p(\A,m) \rightarrow 
    D^p(\A^H,m^*)$ in Proposition \ref{HE1} 
     is surjective if $(\A,m-\delta_H)$ is free. 

\item [(2)]
The restriction map $i_H^*:\Omega^p(\A,m-\delta_H) \rightarrow \Omega^p(\A^H,m^*)$ in Proposition \ref{HE2} is surjective if $(\A,m)$ is free.
\end{itemize}
\label{FST3}
\end{cor}

\section{SPOG results for $\Omega^1(\A)$}\label{sec:SPOG}
It is known (see \cite{A5}) that if $\A$ is a free arrangement, then the
structure of $D(\A')$ is either free or ``nearly'' free in the following
sense.  We say that 
$\A'$ is \emph{strongly plus-one generated} (SPOG)
if there is a minimal set of homogeneous generators 
$\theta_E,\theta_2,\ldots,\theta_\ell,\varphi$ for $D(\A')$
satisfying a relation, unique up to a nonzero scalar multiple, 
\[
\sum_{i=2}^\ell f_i \theta_i+\alpha \varphi=0
\]
where $0 \neq \alpha \in V^*$.
This is to say that the generators of $D(\A')$ have a single relation, and
that at least one of the coefficients in that relation is linear.  See
\cite{A5} for details.

On the other hand, even if 
$\Omega^1(\A)$ is free, $\Omega^1(\A')$ need not even have projective
dimension $1$.  We take a more systematic approach.

\begin{define}
A graded module $S$-module $M$ of rank $\ell$ is \emph{
strongly plus-one generated (SPOG)} if there is a minimal free resolution of the following form:
\[
\begin{tikzcd}
  0 \ar[r] & S[d-1] \ar[rr,"(f_1{,}\ldots{,}f_\ell{,}\alpha)"]  &&
  \oplus_{i=1}^\ell S[d_i] \oplus S[d] \ar[r] & M \ar[r] & 0.
\end{tikzcd}
\]
where $d,d_i \in \Z$ and $0 \neq \alpha \in 
V^*$. We call $\POexp(M):=(-d_1,\ldots,-d_\ell)$ the \emph{exponents} of the 
SPOG module $M$, and $-d$ the \emph{level}.  We call the 
corresponding degree $d$-element a \emph{level element}.
\label{SPOGdef} 
\end{define}

In \cite{A5}, the SPOG property of $D(\A)$ was introduced and studied,
leading to Theorem~\ref{A5} in the introduction.  Let us study the analogue
for logarithmic forms.

\begin{define}
We say that $(\A,m)$ is \emph{dual SPOG} if 
$\Omega^1(\A,m)$ is a SPOG module: that is, if there is a minimal
free resolution
\[
0 \rightarrow S[d-1] 
\rightarrow S[d]\oplus (\oplus_{i=1}^\ell S[d_i]) 
\rightarrow \Omega^1(\A,m) \rightarrow 0,
\]
so that $\POexp(\Omega^1((\A,m))=(-d_1,\ldots,-d_\ell)$ with level $-d$.
\label{dualSPOG}
\end{define}

Abe, Terao and
Wakefield~\cite[Def.\ 2.6]{ATW} define the Poincar\'e polynomial
\[
\pi(\A,m;t)=\sum_{i=0}^\ell b_i(\A,m)t^i
\]
of a multiarrangement via the Solomon--Terao formula, which involves the
Hilbert series of the modules $\Omega^\cdot(\A,m)$.  If $m=\one$, the
main result of \cite{ST87} is that the polynomial is the usual matroid
invariant.  We refer to
\cite{ATW} for details; their definition provides the
correct generalization of certain results for simple arrangements.  In
particular, the main result of Musta{\c{t}}{\u{a} and Schenck~\cite[Thm.\ 4.1]{MS01}
  can be restated for multiarrangements, since the original proof can be
  seen to work without  modification, up to the last step: the
  invocation of Solomon and Terao's
  theorem is simply replaced by the definition from \cite{ATW}.

\begin{theorem}[\cite{MS01}+\cite{ATW}]\label{thm:MS}
  If $(\A,m)$ is a locally free multiarrangement, then
  \[
  \pi(\A,m;t)=\sum_{i=0}^{\ell-1} c_i(\widetilde{\Omega^1(\A,m)})t^i
  \]
  in $\Z[t]/(t^\ell)$.
  \noproof
\end{theorem}

If $(\A,m)$ is a dual SPOG arrangement, the exponents and level of
$\Omega^1(\A,m)$ determine $b_1(\A,m)$ and $b_2(\A,m)$:

\begin{prop}
  Suppose that $(\A,m)$ is dual SPOG with exponents
  $(-d_1,-d_2,\ldots,-d_\ell)$ and 
level $-d$, for some $d_i,d \in \Z$. Then we have 
\begin{eqnarray*}
\abs{m}&=&1+d_1+d_2+\cdots+d_\ell,\\
b_2(\A,m)&=&\sum_{1 \le i < j \le \ell} d_id_j
+\sum_{i=1}^\ell d_i-d+1.
\end{eqnarray*}
\label{betti}
\end{prop}
\begin{proof}
%
  We use the Whitney sum formula applied to the SPOG resolution of $\Omega^1(\A,m)$ to compute total Chern polynomials in $\Z[t]/(t^{\ell})$:
  \begin{align*}
    c_t(\widetilde{\Omega^1(\A,m)}) &= \frac{c_t(\OO_{\P V}(d))    \prod_{i=1}^\ell
    c_t(\OO_{\P V}(d_i))}{c_t(\OO_{\P V}(d-1))}\\
    &=
    1+(1+d_1+d_2+\cdots+d_\ell)t + \Big(\sum_{i<j}d_id_j + \sum_i d_i - d +1
    \Big)t^2 + O(t^3).
  \end{align*}

  The proof of \cite[Prop.\ 5.18]{DS} shows that
  \[
  c_i(\widetilde{(\Omega^1(\A,m))})= c_i(\widetilde{(\Omega^1(\A^X,m)}))
  \]
  for $0\leq i\leq 2$, where $X$ is a generic linear subspace of dimension $3$,
  and $\A^X:=\set{H\cap X\mid H\in\A}$.
  Since rank-$3$ arrangements are locally free, we may now apply
  Theorem~\ref{thm:MS} to $\A^X$.  Noting that $b_1(\A,m)$ is the sum of
  the multiplicities, by \cite[Thm.\ 3.3]{ATW}, the claimed equalities follow.
\end{proof}

\begin{rem}\label{rem:betti}
  If $(\A,m)$ is SPOG with exponents
  $(d_1,d_2,\ldots,d_\ell)$ and level $d$ (that is, if
$D(\A,m)$ is an SPOG module) then by the same arguments as in \cite{A5} and 
Proposition \ref{betti}, we can show that  
\begin{eqnarray*}
\abs{m}&=&d_1+\cdots+d_\ell-1,\\
b_2(\A,m)&=&\sum_{1 \le i < j \le \ell}  d_i d_j-\sum_{i=1}^\ell d_i+d+1.
\end{eqnarray*}
These formulae are similar but not equivalent: 
we refer to Example~\ref{ex:bothSPOGs} at the end of the section.
\end{rem}

We saw that if $\A$ is free, then $\A'$ is SPOG; in the remainder of this
section, we prove the analogous result for logarithmic forms,
that if some deletion $\A'$ is free, then $\A$ is dual SPOG.
We begin with a lemma.
For any multiarrangement $(\A,m)$, let
\[
\omega = \sum_{i=1}^\ell \frac{f_i}Q dx_i
\]
in the usual notation, where $ Q:= Q(\A,m)$ and each $f_i\in S$.
\begin{lemma}
  Let $(\A,m)$ be a multiarrangement and $H\in \A$.  Assume $\alpha_H=x_1$.
  For each hyperplane $X \in \A^H$, let $\eta_X,\omega_X$ be a  
  basis for $\Omega^1(\A_X,m_X)$ for which
$\eta_X \in \Omega^1(\A_X,m_X-\delta_H)$. 
Then 
\[
f_1 \in (\alpha_H,\prod_{X \in \A^H} \alpha_{\widehat{X}}^{e_X}),
\]
where $\widehat{X}$ denotes some hyperplane of $\A$ for which
$X=\widehat{X}\cap H$.

Moreover, if we let $e_X=-\deg \eta_X$ and $d_X=-\deg\omega_X$, then the
polynomial $B:=\prod_{X \in \A^H} \alpha_X^{e_X}$ satisfies $\deg B/Q=-\abs{m}
+\abs{m^*}$. 
\label{Tlemma}
\end{lemma}

\begin{rem}
When $m=\one$, Lemma \ref{Tlemma} is essentially the same as the
Strong Preparation Lemma, \cite[Thm.\ 5.1]{Z2}. Independently, 
Lemma \ref{Tlemma} when $m=\one$ is proved by Terao and the first author in \cite{Tp}.
\end{rem}

\begin{proof}[Proof of Lemma~\ref{Tlemma}.]
Let us fix $X \in \A^H$ and 
assume that $\alpha_H=x_1$ and $X=\{x_1=x_2=0\}$. 
Clearly $Q\Omega^1(\A,m) \subset Q_X\Omega^1(\A_X,m_X)$, where 
$Q_X:=Q(\A_X,m_X)$. Thus 
\[
f_1  \in (g_1^X, h_1 ^X),
\]
where $\eta_X=\sum_{i=1}^\ell (g_i^X/Q_X) dx_i$ and 
$\omega_X=\sum_{i=1}^\ell (h_i^X/Q_X) dx_i$.  
By the choice of $\eta_X$ and $\omega_X$, the polynomial
$g_1^X$ is divisible by $x_1$, and $h_1^X$ is not. 
We may assume that $g_1^X$ and $h_1^X$ are both 
polynomials in $x_1,x_2$. Then $g_1^X=x_1g$ and $h_1^X=x_2^{\abs{m_X}-d_X}
=x_2^{e_X}$ for some $g$, since $e_x+d_X=\abs{m_X}$,
which proves the first assertion.

By definition of $m^*$, it is clear that $\deg B=\abs{m^*}$, which
completes the proof.
\end{proof}

We note that Lemma \ref{Tlemma} is the dual of the following important result,
which we will also need for the main result of this section.

\begin{prop}[\cite{ATW2}, Lemma 3.4]
There is a homogeneous polynomial $B$ of degree $\abs{m}-1-\abs{m^*}$ 
such that 
\[
\theta(\alpha_H) \in (\alpha_H^{m(H)},B)
\]
for any $\theta \in D(\A,m-\delta_H)$.
\label{B}
\end{prop}

In the case where $m = \one$, the module $D(\A,\one-\delta_H)=D(\A \setminus 
\{H\})$ is free or SPOG whenever $\A$ is 
free, by Theorem \ref{A5}. Thus it is natural to ask, if
$\A'=(\A,\one-\delta_H)$ is free, whether $(\A,\one)=\A$ is either 
free or dual SPOG, and this is exactly our Theorem \ref{SPOGmain}.

It is also natural to ask whether the same holds in the generality of
multiarrangements.  The multi-version of Theorem \ref{A5} could not
be proven in \cite{A5}, since 
the proof there used Ziegler restriction, which is intrinsic to simple
arrangements.
With the technique here, we are now able to determine completely
the relationship between freeness, (dual) SPOG properties, and
addition/deletion operations.  Theorems~\ref{A5}, \ref{SPOGmain} are
special cases of the following main result.

\begin{theorem}~
\begin{itemize}
\item[(1)]    
Suppose $(\A,m)$ is free. Then $(\A,m-\delta_H)$ is either free or SPOG.
\item[(2)]
Suppose $(\A,m-\delta_H)$ is free. Then $(\A,m)$ is either free or dual SPOG.
\end{itemize}
\label{multimain}
\end{theorem}

\begin{proof}[Proof of (1):]
Assume that $(\A,m-\delta_H)$ is not free. 
Recall that, by Lemma \ref{ID}, there is an 
identification $\Omega^{\ell-1}(\A,m-\delta_H) \cong D(\A,m-\delta_H)$ by 
the correspondence 
\[
\sum_{i=1}^\ell \displaystyle \frac{f_i}{Q(\A,m-\delta_H)} dx^i \mapsto 
\sum_{i=1}^\ell f_i \partial_{x_i}.
\]
Here again, $dx^i:= \wedge_{j\neq i} dx_j$.
Since $\rank\A^H=\ell-1$, it is the case that
\[
\Omega^{\ell-1}(\A^H,m^*)=(1/Q(\A^H,m^*)) dx,
\]
where 
\[
dx:=\wedge_{i=2}^\ell dx_j
\]
and $H:=\ker x_1$. 
By Proposition \ref{HE2} and Corollary \ref{FST3}, we have the exact sequence
\[
0 \rightarrow 
\Omega^{\ell-1}(\A,m) 
\stackrel{\cdot \alpha_H}{\rightarrow}
\Omega^{\ell-1}(\A,m-\delta_H)
\stackrel{i_H^*}{\rightarrow}
\Omega^{\ell-1}(\A^H,m^*)
\rightarrow 0.
\]
Hence there is some
$\omega \in \Omega^{\ell-1}(\A,m-\delta_H)$ such that $i_H^*(\omega)
=(1/Q(\A^H,m^*))  dx$, thus $\deg 
\omega=-\abs{m^*}$. 
Moreover, in the expression 
\[
\omega=\sum_{i=1}^\ell (f_i/Q(\A,m-\delta_H)) dx^i,
\]
the fact that 
$i_H^*(\omega) \neq 0$ shows $x_1 \nmid f_1$. Under the identification
with derivations, $\omega$ corresponds to 
\[
\theta=\sum_{i=1}^{\ell} f_i
\partial_{x_i} 
\in D(\A,m-\delta_H)
\]
with $x_1 \nmid f_1$. Since 
$\deg \theta=\abs{m}-1-\abs{m^*}$ and $\theta 
\not \in D(\A,m)$, Proposition \ref{B} shows that, 
for all $\varphi \in D(\A,m-\delta_H)$, there is $f \in S$ such that 
$\varphi-f\theta \in D(\A,m)=
\oplus_{i=1}^\ell S \theta_i$, where $\theta_1,
\ldots,\theta_\ell$ form a basis for $D(\A,m)$. 
Hence 
\[
D(\A,m-\delta_H)=\langle \theta_1,\ldots,\theta_\ell,
\theta\rangle.
\]
Since $(\A,m-\delta_H)$ is not free, this is a 
minimal set of generators.
Thus it suffices to determine the relation among 
this set of generators. Since $\alpha_H \theta \in D(\A,m)$, there is the unique relation 
in degree $d+1:=\abs{m}-\abs{m^*}$. This is to say that $D(\A,m-\delta_H)$ is
SPOG with 
$\POexp(D(\A,m-\delta_H))=\exp(\A,m)$ and level $\abs{m}-1-\abs{m^*}$.
\end{proof}
\begin{proof}[Proof of (2):]
Assume that $(\A,m)$ 
is not free.
Recall that $\Omega^1(\A,m) \cong D^{\ell-1}(\A,m)$ by 
Lemma \ref{ID}. 
Let $\partial^i:= \wedge_{j\neq i} \partial_{x_j}$.
Now consider the exact sequence
\[
0 \rightarrow 
D^{\ell-1}(\A,m-\delta_H) 
\stackrel{\cdot \alpha_H}{\rightarrow}
D^{\ell-1}(\A,m)
\stackrel{\rho^H_{\ell-1}}{\rightarrow}
D^{\ell-1}(\A^H,m^*)
\]
in Proposition \ref{HE1}. 
Since $\dim_\K H=\ell-1$, the module $D^{\ell-1}(\A^H,m^*)$ has rank $1$ and
is generated by $Q(\A^H,m^*) \partial$, 
where 
\[
\partial:=\wedge_{i=2}^\ell \partial_{x_i}
\]
assuming again that $H=\ker x_1$. 
By Proposition \ref{HE1}, 
Corollary \ref{FST3} and the freeness of $(\A,m-\delta_H)$, we have the exact sequence
\[
0 \rightarrow 
D^{\ell-1}(\A,m-\delta_H) 
\stackrel{\cdot \alpha_H}{\rightarrow}
D^{\ell-1}(\A,m)
\stackrel{\rho^H}{\rightarrow}
D^{\ell-1}(\A^H,m^*)
\rightarrow 0.
\]
In particular, $\rho^H$ is surjective, so 
there is some $\theta \in D^{\ell-1}(\A,m)$ such that $\rho^H(\theta)
=Q(\A^H,m^*)  \partial$, and it has $\deg 
\theta=\abs{m^*}$.
Since $\theta(x_1,x_{i_2},\ldots,x_{i_{\ell-1}}) \in S\alpha_H^{m(H)}$ for all 
$\set{i_2,\ldots,i_{\ell-1}} \subset \set{2,\ldots,\ell}$, if we express
\[
\theta=\sum_{i=1}^\ell f_i  \partial^i,
\]
for some $f_i$'s, we have $x_1^{m(H)} \mid f_i$ unless $i=1$. Since 
$\rho^H(\theta) \neq 0$, we also see $x_1 \nmid f_1$. Under the correspondence
with forms (Lemma~\ref{ID}), $\theta$ corresponds to 
\[
\omega=\sum_{i=1}^{\ell} \displaystyle \frac{f_i}{Q(\A,m)}dx_i 
\in \Omega^1(\A,m)
\]
with $x_1 \nmid f_1$. Since 
$\deg \omega=-|m|+|m^*|=\deg B/Q(\A,m)$ and $\omega$ has a pole along $H$, 
Lemma \ref{Tlemma} shows that, for all $\eta \in \Omega^1(\A,m)$, there is $f \in S$ such that 
$\omega-f\eta\in \Omega^1(\A,m-\delta_H)=
\oplus_{i=1}^\ell S \omega_i$. 
Hence 
\[
\Omega^1(\A,m)=\langle w_1,\ldots,\omega_\ell,
\omega\rangle.
\]
Since $(\A,m)$ 
is not free, they form a minimal set of generators. 
Thus it suffices to determine the relation among these 
generators. Since $\alpha_H \omega \in \Omega^1(\A,m-\delta_H)$, 
there is the unique relation 
in degree $-\abs{m}+1+\abs{m^*}:=-d$. Once again,
$\Omega^1(\A,m)$ is SPOG with $\POexp(D(\A,m))=-\exp(\A,m-
\delta_H)$ and level $-\abs{m}+\abs{m^*}$.
\end{proof}

\begin{example}\label{ex:bothSPOGs}
  Consider the graphic arrangement $\A=\A({\mathcal W}_4)$ of
  $8$ hyperplanes given by the ``wheel'' graph
  \[
 {\mathcal W}_4=
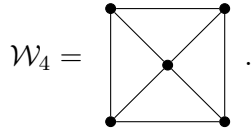
\begin{tikzpicture}[baseline=(current bounding box.center),scale=0.5,
bend angle=25,vertex/.style={circle,draw,inner sep=1.3pt,fill=black}]

\node[vertex] (a) at (-0.5,-0.5) {};
\node[vertex] (b) at (2.5,-0.5) {};
\node[vertex] (c) at (2.5,2.5) {};
\node[vertex] (d) at (-0.5,2.5) {};
\node[vertex] (e) at (1,1) {};

\draw (a) to node[below] {} (b);
\draw (b) to node[right] {} (c);
\draw (c) to node[above] {} (d);
\draw (d) to node[left] {} (a);
\draw (a) to node[above] {} (e);
\draw (b) to node[above] {} (e);
\draw (c) to node[below] {} (e);
\draw (d) to node[below] {} (e);
\end{tikzpicture}.
  \]
By adding an edge, one obtains a deletion of the graphic arrangement $K_5$,
which
is free.  By deleting any of the last four outer edges, one obtains a
chordal graphic arrangement, which is supersolvable, hence also free.
However, $\A$ is not free, so by
Theorem~\ref{multimain}, it is both SPOG and dual SPOG.  A
Macaulay2~\cite{M2} computation shows $D(\A)$ has exponents
$(1,2,3,3)$ with level $d=3$ and $\Omega^1(\A)$ has exponents $(-2,-2,-2,-1)$
and level $d=-2$.  Using Proposition~\ref{betti} we compute
$b_1(\A)=8=(2+2+2+1)+1$ and
\begin{eqnarray*}
  b_2(\A) &=& (4+4+4+2+2+2)+(2+2+2+1)-2+1\\
  &=& 24
\end{eqnarray*}
while Remark~\ref{rem:betti},
\begin{eqnarray*}
  b_1(\A) &=& (1+2+3+3)-1\quad\text{and}\\
  b_2(\A) &=& (2+3+3+6+6+9)-(1+2+3+3)+3+1\\
  &=& 24\quad\text{again.}
\end{eqnarray*}
\end{example}

As an application to conclude this section, we give a shorter proof of the first author's
Addition Theorem~\cite{A6}: namely, if one adds a hyperplane to a free
arrangement, the freeness of the result is a combinatorial property:

\begin{theorem}[\cite{A6}, Theorem 1.4]
Let $H \in \A, \A':=\A 
\setminus \{H\}$. Assume that $\A'$ is free. Then $\A$ is free if and only if 
$|\A'_X|-|\A^H_X|$ is a root of $\chi_0(\A'_X;t)$ for all $X \in L(\A^H)$.
\label{addcomb}
\end{theorem}

\begin{proof}
The ``only if '' part is clear. Let us prove the 
``if'' part. The statement is clear if $\ell \le 3$, by, e.g., \cite{A}.
Thus we can use the induction on $\ell$. 
By the induction hypothesis and the assumption, we may assume that $\A$ is locally free: that is,
$\A_X$ is free for all $X \in L(\A^H) \setminus \{0\}$. 
Assume that $\A$ is not free. Then by Theorem \ref{SPOGmain} , $\Omega^1(\A)$ is SPOG. 
Let 
\[
0 \rightarrow 
S[d-1] \stackrel{f}{\rightarrow} \oplus_{i=1}^\ell S[d_i] \oplus 
S[d] \rightarrow \Omega^1(\A) \rightarrow 0
\]
be a minimal free resolution. Here we assume that 
$-1=d_1\ge -d_2 \ge \cdots \ge -d_\ell$. 
We may assume that $\A'$ is essential, thus $-1>-d_2 \ge \cdots \ge -d_\ell$. 
Since the generating set is minimal, 
we may assume that the coefficient of the 
degree-$(-1)$ element in the relation is zero by the direct sum decomposition. 
So 
\[
\sum_{i=2}^\ell g_i \omega_i +\alpha_H \omega=0,
\]
where $d\alpha_L/\alpha_L, \omega_2,\ldots,\omega_\ell$ 
form a basis for $\Omega^1(\A')$ and $\omega \in 
\Omega^1(\A)$ a level element. By the assumption 
on $|\A'|-|\A^H|$, 
$\deg \omega=\deg \omega_i$ for some $i \ge 2$. 
Thus $g_j=0$ for all $j$ with $-d_j<-d_i$. If there is some $k$ with $-d_k=-d$ and 
$g_k \neq 0$, then we may replace $\omega_k$ by $\alpha_H \omega$ and $\omega_1,
\omega_2,\ldots,\omega_{k-1},\omega,
\omega_{k+1},
\ldots,
\omega_\ell$ form 
a basis 
for $\Omega^1(\A)$. Thus such $g_k=0$. Namely, if we replace $d_i$ in such a way that 
$d_i<d_{i+1}$, then $g_j=0$ for all $j \le i$.
Then let 
$\pp$ be a prime ideal containing $g_{i+1},\ldots,g_{\ell},\alpha_H$. Then 
$\pp$ determines a non-empty point in $\Proj(S)$, and at this point 
$f_\pp=0$, so taking the residue field of the minimal free resolution above, 
\[
k_{\pp}^{\ell+1} \cong \Omega^1(\A) \otimes k_{\pp} 
\cong k_{\pp}^\ell
\]
since $\A$ is locally free, a contradiction.
\end{proof}

\section{Dual morphisms}\label{sec:LES}
In this section, we compare the behavior of the Euler restriction
maps of logarithmic derivations and logarithmic forms, respectively.
First let us prove the following.

\begin{prop}
For all $\oS$-modules $M$ and all $p\geq0$,
\[
\Ext_{\oS}^p(M,\oS[1]) \cong 
\Ext_{S}^{p+1}(M,S)
\]
as $S$-modules.
\label{prop:dual}
In particular, we have isomorphisms
\[
\Ext^1_{\oS}(D(\A^H),\oS[1])\cong \Omega^1(\A^H) \quad \text{and} \quad
\Ext^1_{\oS}(\Omega^1(\A^H),\oS[1])\cong D(\A^H).
\]
\end{prop}
\begin{proof}
We apply the base change spectral sequence, using the fact that 
$0\rightarrow S[-1]\stackrel{\alpha_H}{\rightarrow}S\rightarrow \oS\rightarrow 0$ is a free resolution of $\oS$ over $S$.
\end{proof}

Once again, let $(\A,m)$ be a multiarrangement for which $m_H>0$, for some
$H\in\A$.  Let $m'=m-\delta_H$, and $m^*$ the Euler multiplicity on $\A^H$.
Since the
Euler restriction map $\rho^H$ need not be surjective, we will denote its
image by $C_{D,H}(m)$ or just $C_{D}$ if the choice of $H$ and $m$ are
understood.  Then we have a short exact sequence of graded $S$-modules
\begin{equation}\label{eq:Dseq}
  0\rightarrow D(\A,m')[-1] \stackrel{\cdot\alpha_H}{\rightarrow}
  D(\A,m) \rightarrow C_D \rightarrow 0
\end{equation}
and an inclusion of $\oS$-modules
$C_D\hookrightarrow D(\A^H,m^*)$ which is an isomorphism precisely
when the restriction map $\rho^H$ is surjective.  The module $C_D$ consists
of the logarithmic
derivations on $(\A^H,m^*)$ which can be lifted to $(\A,m')$.

Dually, let $C_{\Omega,H}(m)$
or just $C_{\Omega}$ denote the image of the restriction map
$i_H^*\colon \Omega^1(\A,m')\to \Omega^1(\A^H,m^*)$.  This 
gives a short exact sequence of graded $S$-modules
\begin{equation}\label{eq:Oseq}
  0\rightarrow \Omega^1(\A,m)[-1] \stackrel{\cdot\alpha_H}{\rightarrow}
  \Omega^1(\A,m') \rightarrow C_\Omega \rightarrow 0
\end{equation}
and an inclusion of $\oS$-modules
$C_{\Omega}\hookrightarrow \Omega^1(\A^H,m^*)$ which is an
isomorphism precisely when $i_H^*$ is surjective.  The module $C_{\Omega}$
consists of the logarithmic
$1$-forms on $\A^H$ which are restrictions of logarithmic $1$-forms on $(\A,m)$.

Keeping track of local information leads to the following definitions.

\begin{define}\label{def:nonfree_locus}
  Let $\A$ be a central arrangement of rank $\ell$ and $m\in \Z^{\A}_{\geq0}$.
  \begin{enumerate}
    \item Let 
      \[
      \NF(\A,m):=\set{\pp\in(\spec S)
        \colon D(\A,m)_{\pp}\text{~is not free}},
      \]
      the nonfree locus of $(\A,m)$.
    \item For each $H\in \A$, let
      \[
      \ND(\A,m,H):=\set{\pp\in(\spec \oS)
        \colon (C_D)_{\pp}\not\cong D(\A^H,m^*)_{\pp}}.
      \]
      We will call this the \emph{$D$-exceptional set of $(\A,m,H)$}.
    \item For each $H\in \A$, let
      \[
      \NO(\A,m,H):=\set{\pp\in(\spec \oS)
        \colon (C_{\Omega})_{\pp}\not\cong \Omega^1(\A^H,m^*)_{\pp}}.
      \]
      We will call this the \emph{$\Omega$-exceptional set of $(\A,m,H)$}.
  \end{enumerate}
\end{define}

\begin{rem}
  The loci above are all Zariski closed sets.  For example,
  for any triple $(\A,m,H)$, we have $\pp\in \ND(\A,m,H)$
  if and only if $\pp$ is in the support of the quotient of
  $C_D\hookrightarrow D(\A^H,m^*)$.   
\end{rem}
\begin{rem}
  Since $X\mapsto D(\A_X)$ is a local functor, in the sense of
  Orlik--Terao~\cite[Def.\ 4.121]{OT}, the non-free locus is a union of
  subspaces from $L(\A)$.  The map $\NF(\A,m)\to L(\A)$ given by
  \[
  \pp \mapsto \bigcap_{H\in\A\colon H\supseteq V(\pp)} H
  \]
  has an upward-closed image in $L(\A)$.  
\end{rem}
\begin{rem}
Clearly $\pp\in\NF(\A)$ if and only if the arrangement $\A_{X(\pp)}$ is not
free.
Since the multiarrangement $(\A_X,m_X)$ is free for all flats $X$ of rank
$\leq 2$ and multiplicities~\cite[Prop.~1.1]{ATW2}, the set $\NF(\A,m)$
has codimension at least $3$ in $\spec S$.
\end{rem}

\begin{prop}\label{prop:mostlyexact}
  For any triple $(\A,m,H)$ we have
\[
\ND(\A,m,H)\subseteq \NF(\A,m') \quad\text{and}\quad
\NO(\A,m,H)\subseteq \NF(\A,m).
\]
In particular, $\ND(\A,m,H)$ and $\NO(\A,m,H)$ both have codimension at
least $2$ in $\spec\oS$.
\end{prop}
\begin{proof}
  We apply Theorems~\ref{thm:FST} and \ref{thm:FST2}, respectively, to the
  arrangements $(\A_{X(\pp)},m_{X(\pp)})$, for each prime ideal $\pp$.
\end{proof}
  
Although the
inclusions $C_D\hookrightarrow D(\A^H,m^*)$ and $C_{\Omega}\hookrightarrow
\Omega^1(\A^H,m^*)$ may be strict, their duals are isomorphisms.
Let $-^\vee=\Hom_\oS(-,\oS)$.  For each $m\in\Z^\A_{\geq0}$, 
$c_D(m):=c_D(\A,m,H)$ and $c_{\Omega}(m):=c_{\Omega}(\A,m,H)$
denote the codimensions of $\ND(\A,m,H)$ and $\NO(\A,m,H)$
in $\spec\oS$, respectively, with the convention that $\codim\emptyset=\infty$.

\begin{prop}\label{prop:coker_dual}
  For any triple $(\A,m,H)$, we have
  \[
  \Omega^1(\A^H,m^*)\cong C_D^\vee, \quad\text{and}\quad D(\A^H,m^*)\cong
  C_{\Omega}^\vee.
  \]
  More precisely, the natural maps
  \begin{align*}
  & \Ext^p_{\oS}(D(\A^H,m^*),\oS)\rightarrow \Ext^p_{\oS}(C_D,\oS)
  \quad\text{and}\\
  & \Ext^p_{\oS}(\Omega^1(\A^H,m^*),\oS)\rightarrow  \Ext^p_{\oS}(C_\Omega,\oS)
  \end{align*}
  are isomorphisms for $p\leq c_D(m)-2$ and $p\leq
  c_{\Omega}(m)-2$, respectively.  The
  maps are monomorphisms for $p=c_D(m)-1$ and $p=c_{\Omega}(m)-1$, respectively.
\end{prop}
\begin{proof}
  Let $G$ denote the quotient,
\begin{equation}\label{eq:CDvsD}
\begin{tikzcd}[column sep=small]
0\ar[r] & C_D \ar[r] & D(\A^H,m^*)\ar[r] & G\ar[r] & 0.
\end{tikzcd}
\end{equation}
By definition, localizations of the module $G$ are zero outside of 
$\ND(\A,m,H)$.  Since $\oS$ is Cohen--Macaulay, this implies that
$\Ext^p_{\oS}(G,\oS)=0$ for $p<\codim\ND(\A,m,H)$: see,
for example, \cite[17.1]{Mat86}.

The claims for $C_D$ then follow by examining the long exact sequence
of $\Ext^\cdot_{\oS}(-,\oS)$ applied to \eqref{eq:CDvsD}, noting that
for $p=0$, $D(\A^H,m^*)^\vee\cong \Omega^1(\A^H,m^*)$.

The corresponding properties of $C_{\Omega}$ are proven in the same way.
\end{proof}

Using the sequences \eqref{eq:Dseq} and \eqref{eq:Oseq}, this 
gives a measure of when the Euler restriction maps are surjective.
\begin{theorem}\label{thm:les}
  For any arrangement triple $(\A,m,H)$,
  the Euler sequences extend to the following exact sequences,
  for $p=c_{\Omega}(m)-1$ and $p=c_D(m)-1$, respectively:
  \begin{align*}
    0\rightarrow & D(\A,m')[-1]\stackrel{\cdot\alpha_H}{\rightarrow}
    D(\A,m)    \rightarrow D(\A^H,m^*)\rightarrow\\
    \rightarrow & \Ext^1_S(\Omega^1(\A,m'),S)[-1] \rightarrow
    \Ext^1_S(\Omega^1(\A,m),S) \rightarrow
    \Ext^1_{\oS}(\Omega^1(\A^H,m^*),\oS)\rightarrow\\
 \rightarrow  & \cdots \rightarrow \\
 \rightarrow & \Ext^p_S(\Omega^1(\A,m'),S)[-1] \rightarrow
 \Ext^p_S(\Omega^1(\A,m),S) \rightarrow \Ext^p_{\oS}(C_{\Omega},\oS)
  \end{align*}
  and
  \begin{align*}
    0\rightarrow & \Omega^1(\A,m)[-1]\stackrel{\cdot\alpha_H}{\rightarrow}
    \Omega^1(\A,m') \rightarrow \Omega^1(\A^H,m^*)\rightarrow\\
    \rightarrow & \Ext^1_S(D(\A,m),S)[-1] \rightarrow \Ext^1_S(D(\A,m'),S)
    \rightarrow \Ext^1_{\oS}(D(\A^H,m^*),\oS) \rightarrow \\
 \rightarrow  & \cdots \rightarrow \\
 \rightarrow & \Ext^p_S(D(\A,m),S)[-1] \rightarrow\Ext^p_S(D(\A,m'),S)
 \rightarrow \Ext^p_{\oS}(C_{D},\oS).
  \end{align*}
\end{theorem}
\begin{proof}
  To prove the first statement, we apply $\Hom_S(-,S)$ to the short exact
  sequence \eqref{eq:Oseq} and consider the long exact sequence
  \begin{align*}
    0\cong& \Hom_S(C_\Omega,S)\rightarrow \Hom_S(\Omega^1(\A,m'),S)
    \rightarrow  \Hom_S(\Omega^1(\A,m),S)[1]\rightarrow\\
    \rightarrow & \Ext^1_S(C_{\Omega},S) \rightarrow
    \Ext^1_S(\Omega^1(\A,m'),S)
    \rightarrow \Ext^1_S(\Omega^1(\A,m),S)[1]\rightarrow \\
    \rightarrow & \Ext^2_S(C_{\Omega},S)
  \rightarrow \cdots
  \end{align*}
  We use the isomorphisms $\Ext^{i+1}_S(C_\Omega,S)\cong
  \Ext^i_{\oS}(C_\Omega,\oS)[1]$ for each $i$, by Proposition~\ref{prop:dual},
  and $\Ext^i_{\oS}(C_\Omega,\oS)
  \cong \Ext^i_{\oS}(\Omega^1(\A^H,m^*),\oS)$ for $i\leq c_{\Omega}-2$, by
  Proposition~\ref{prop:coker_dual}, to obtain the
  first exact sequence shown above.  For the second statement, we start
  with \eqref{eq:Dseq} instead.
\end{proof}
We recall that, if the deletion
$(\A,m')$ is free, the Euler sequence of derivations
is right exact (Theorem~\ref{thm:FST}).
The Euler sequence of $1$-forms \eqref{E1},
on the other hand, extends to a four-term exact sequence
\[
\begin{tikzcd}[column sep=tiny]
  0\ar[r] & \Omega^1(\A,m) \ar[r] & \Omega^1(\A,m')\ar[r] &
  \Omega^1(\A^H,m^*)\ar[r] &
  \Ext^1_S(D(\A,m),S)\ar[r] & 0.
\end{tikzcd}
\]
As usual, the dual observation is obtained by exchanging $m'$ and $m$.

We arrive at a slight refinement of Proposition~\ref{prop:mostlyexact}.
\begin{cor}
  For any arrangement triple $(\A,m,H)$, we have
  \begin{align*}
    \ND(\A,m,H)&\subseteq\supp\Ext^1_S(\Omega^1(\A,m'),S),
    \quad \text{and}\\
  \NO(\A,m,H)&\subseteq\supp\Ext^1_S(D(\A,m),S).
  \end{align*}
  Equality occurs on the top if $(\A,m)$ is free, and on the bottom if
  $(\A,m')$ is free.  
\end{cor}
In particular, if an arrangement $(\A,m)$ contains hyperplanes $H_1$ and $H_2$
for which both deletions are free, then $\NO(\A,m,H_1)=
\NO(\A,m,H_2)$: that is, all free deletions produce the same
$\Omega$-exceptional set.  Since $\NO(\A,m,H)\subseteq H=\spec(\oS)$ for each
$H$, if $(\A,m-\delta_H)$ is free, then
\[
\NO(\A,m,H)\subseteq\bigcap_{
  \substack{K\in\A\colon \\ (A,m-\delta_K)\text{~is free}}} K.
\]

\begin{example}
  Consider the simple arrangement of $10$ hyperplanes in $\CC^4$ defined by the
  columns of the matrix
\[
\setcounter{MaxMatrixCols}{20}
\begin{pmatrix}
1 & 0 & 0 & 1 & 0 & 0 & 1 & 1 & 0 & 1\\
0 & 1 & 0 & 0 & 1 & 0 & 1 & 0 & 1 & 1\\
0 & 0 & 1 & 0 & 0 & 1 & 0 & 1 & 1 & 1\\
0 & 0 & 0 & 1 & 1 & 1 & 1 & 1 & 1 & 1\\
\end{pmatrix}.
\]
$\A$ is free, so by Theorem~\ref{thm:les}, we have
$C_{\Omega}\cong \Omega^1(\A^H)$ for each $H$: that is, the $\Omega$-exceptional
set is empty for each $H$.

Numbering the hyperplanes in order, computation shows
\[
\begin{array}{|l|l|l|}
  i & \NF(\A,\one-\delta_{H_i}) & \ND(\A,\one,H_i)\\ \hline
  1,2,3,7,8,9 & \text{$1$-dimensional} & \set{0}\\
  4,5,6 & \emptyset & \emptyset\\
  10 & \set{0} & \emptyset
\end{array}
\]
For example for $H=H_{10}$, we see that although the deletion $\A'$ is not
free, the restriction map in the Euler sequence \eqref{eq:euler} is surjective.  The long exact sequence
  above shows that $\Ext^1_S(\Omega^1(\A'),S)=0$, and
  \[
  \Ext^1_{\oS}(\Omega^1(\A^H),\oS)[1]\cong \Ext^2_S(\Omega^1(\A'),S).
  \]
Since $\A^H$ is not free, these must both be nonzero.
\end{example}

\bibliographystyle{plain}

\end{document}